\documentclass[11pt,reqno]{amsart}
\usepackage{amsmath, amssymb, amsfonts, xcolor}
\usepackage{hyperref}
\numberwithin{equation}{section}
\usepackage{amsthm}
\newtheorem{thm}{Theorem}[section]

\newtheorem{prop}[thm]{Proposition}
\newtheorem{lem}[thm]{Lemma}
\newtheorem{conj}[thm]{Conjecture}
\newtheorem{dfn}[thm]{Definition}

\newtheorem{remark}[thm]{Remark}
\newtheorem{cor}[thm]{Corollary}

\usepackage[a4paper, top = 1.2in, bottom = 1.2in, left = 1.2in, right = 1.2in]{geometry}
\usepackage{enumitem}
\newlist{steps}{enumerate}{1}
\setlist[steps, 1]{label = Step \arabic*:}

\numberwithin{equation}{section}

\newcommand{\C}{\mathbb{C}}
\newcommand{\F}{\mathbb{F}}

\newcommand{\Q}{\mathbb{Q}}

\newcommand{\Z}{\mathbb{Z}}

\newcommand{\mcO}{\mathcal{O}}

\newcommand{\mfb}{\mathfrak{b}}

\newcommand{\mfm}{\mathfrak{m}}
\newcommand{\mfn}{\mathfrak{n}}
\newcommand{\mfr}{\mathfrak{r}}

\newcommand{\mfp}{\mathfrak{p}}
\newcommand{\mfq}{\mathfrak{q}}
\newcommand{\mfP}{\mathfrak{P}}

\newcommand{\GL}{\mathrm{GL}}

\newcommand{\Gal}{\mathrm{Gal}}

\def\1{1\!\!1}

\title[Generalized Fermat equation of signature $(2p, 2q, r)$]{Effective Generalized Fermat equation of signature $(2p, 2q, r)$ with odd narrow class number}

\author[S. Sahoo]{Satyabrat Sahoo}
\address[S. Sahoo]{Yau Mathematical Sciences Center, Tsinghua University, Beijing 100084, China}
\email{satyabrat.sahoo.94@gmail.com}

\date{\today}

\begin{document}
	\begin{abstract}
		Fix a rational prime $r \geq 5$. In this article, we study the integer solutions of the generalized Fermat equation of signature $(2p,2q,r)$, namely $x^{2p}+y^{2q}=z^r$, where the primes $p,q \geq 5$ are varying. For each rational prime $r \geq 5$, we first establish a condition on the solutions of the $S$-unit equation over  $\Q(\zeta_r+ \zeta_r^{-1})$ such that there exists a constant $V_{r}>0$ (depending on $r$) for which the equation $x^{2p}+y^{2q}=z^r$ with $p,q \geq V_r$ has no non-trivial primitive integer solutions. Then for each rational prime $r \geq 2$, we prove that every elliptic curve over $\Q(\zeta_r+ \zeta_r^{-1})$ is modular. As an application of this, we prove that the above constant $V_r$ is effectively computable. Finally, we provide a criterion for $r$ such that the equation $x^{2p}+y^{2q}=z^r$ with $p,q \geq V_r$ has no non-trivial primitive integer solutions when the narrow class number of $\Q(\zeta_r+ \zeta_r^{-1})$ is odd. 
	\end{abstract}
	
	\maketitle
	
	\section{Introduction}
	The study of Diophantine equations, namely Fermat-type equations over totally real number fields via the modular method is one of the most active and exciting areas in number theory. Wiles's proof of Fermat's Last Theorem using the modular method gave birth a new way to solve Diophantine equations. Since then, remarkable progress has been achieved in the study of generalized Fermat equation,  
	\begin{equation}
		\label{p,q,r}
		Ax^p+By^q=Cz^r, \ p,q,r \in \Z_{\geq 2} \text{ with } \frac{1}{p} +\frac{1}{q}+ \frac{1}{r} <1,
	\end{equation}
	where $A,B,C \in \Z \setminus \{0\}$ are coprime. We say $(p,q,r)$ as the signature of \eqref{p,q,r}. We say a solution $(a,b,c) \in \Z^3$ to the equation \eqref{p,q,r} is non-trivial if $abc \neq 0$, and primitive if $a,b,c$ are pairwise coprime.
	Then, the following conjecture is known for equation~\eqref{p,q,r}. 
	\begin{conj}
		\label{DG conj}
		Fix $ A,B,C \in \Z \setminus \{0\}$ are coprime. Then, for all prime exponents $ p,q,r$ with $\frac{1}{p} +\frac{1}{q}+ \frac{1}{r} <1$, the equation~\eqref{p,q,r} has only finitely many non-trivial primitive integer solutions (here the solutions like $1^p+ 2^3=3^2$ counted only once for all primes $p$).
	\end{conj}
	In \cite{DG95}, Darmon and Granville proved a partial result towards Conjecture~\ref{DG conj}.
	\begin{thm}{\cite[Theorem 2]{DG95}}
		For fixed coprime integers $ A,B,C \in \Z \setminus \{0\}$ and for fixed primes $ p,q,r$ with $\frac{1}{p} +\frac{1}{q}+ \frac{1}{r} <1$, the equation~\eqref{p,q,r} has only finitely many non-trivial primitive integer solutions.
	\end{thm}
	Moreover, Conjecture~\ref{DG conj} has been proved for many signatures, including the signatures $(p,p,p), (p,p,2), (p,p,3)$ (cf. \S\ref{literature survey} for more details). 
	In this article, we study the solutions of the generalized Fermat equation of signature $(2p,2q,r)$, i.e.,
	\begin{equation}
		\label{2p,2q,r}
		x^{2p}+y^{2q}=z^r,
	\end{equation} 
	where the primes $p,q \geq 5$ are varying and the prime $r \geq 5$ is fixed. First, we give the literature before stating the main results of this article.
	\subsection{Literature Survey:}
	\label{literature survey}
	Let $r,p$ denote the rational primes and $m$ denotes a positive integer.
	In this subsection, we give a literature survey for the generalized Fermat equation~\eqref{p,q,r} of signature $(p,p,p), (p,p,2), (p,p,3), (p,p,r)$ and $(2p,2q,r)$.
	\subsection*{Generalized Fermat equation of signature $(p,p,p)$:}
	In \cite{W95}, Wiles proved that the Fermat equation $x^p+y^p=z^p$ with  $p \geq 3$ has no non-trivial primitive integer solutions.
	In~\cite{DM97}, Darmon and Merel proved that the equation $x^p+y^p=2z^p$ with $p \geq 3$ has no non-trivial primitive integer solutions. In~\cite{R97}, Ribet proved that the equation $x^p+y^p=2^mz^p$ has no non-trivial primitive integer solutions for $2\leq m <p$. Finally, in \cite{K97}, Kraus studied the integer solutions of the generalized Fermat equation $Ax^p+By^p=Cz^p$ for various choices of non-zero integers $A,B,C$.
	Over totally real number fields $K$, the solutions of $Ax^p+By^p=Cz^p$ with $A,B,C \in \mcO_K\setminus \{0\}$ has been studied by \cite{FS15}, \cite{D16}, \cite{KS23 Diophantine1} and \cite{S24 GFLT}.
	\subsection*{Generalized Fermat equation of signature $(p,p,2)$:}
	In~\cite{DM97}, Darmon and Merel first proved that the equation $x^p+y^p=z^2$ with $p \geq 3$ has no non-trivial primitive integer solutions. In~\cite{I03}, Ivorra studied the integer solutions of the equation $x^p+2^my^p=z^2$ and $x^p+2^my^p=2z^2$ for $p \geq 5$ and $0 \leq m < p$. In ~\cite{S03}, Siksek proved that the only non-trivial primitive integer solutions to the equation $x^p+2^my^p=z^2$ with $p \geq 5$ and $m \geq 2$ are when $m=3,\ x=y=1, \ z=\pm 3$. In \cite{BS04} (respectively, \cite{IK06}), Bennet and Skinner (respectively, Ivorra and Kraus) studied the integer solutions of the generalized Fermat equation $Ax^p+By^p=Cz^2$ for various choices of non-zero integers $A,B,C $. 
	Over totally real number fields $K$, the solutions of $Ax^p+By^p=Cz^2$ with $A,B,C \in \mcO_K\setminus \{0\}$ has been studied by \cite{IKO20}, \cite{M22}, \cite{KS23 Diophantine1} and \cite{KS23 Diophantine2}.
	\subsection*{Generalized Fermat equation of signature $(p,p,3)$:}
	In~\cite{DM97}, Darmon and Merel proved that the equation $x^p+y^p=z^3$ with $p \geq 3$ has no non-trivial primitive integer solutions. In \cite{BVY04}, Bennett et al. studied the integer solutions of the generalized Fermat equation $Ax^p+By^p=Cz^3$, for various choices of non-zero integers $A,B,C $. Over totally real number fields $K$, the solutions of $Ax^p+By^p=Cz^3$ with $A,B,C \in \mcO_K\setminus \{0\}$ has been studied by \cite{IKO23},\cite{M22} and \cite{SK25}.
	
	\subsection*{Generalized Fermat equation of signature $(p,p,r)$:}
	We now give the literature for the generalized Fermat equation \eqref{p,q,r} of signature $(p,p,r)$ with $r \geq 5$. 
	In \cite{CK22}, Chen and Koutsianas proved that the equation $x^p+y^p=z^5$ with $p \geq 3$ has no non-trivial primitive integer solutions $(a,b,c)$ with $2 \nmid ab$ and $5 | ab$ or $2 | ab$ and $5 \nmid ab$. Recently in \cite{Azon25}, Azon proved that the equation $7x^p+2^i5^jy^p=3z^5$ with $i \in \{1,2,3,4\}$ and $j \in \{3,4\}$ (respectively, $7x^p+y^p=3z^5$) with $p >71$ has no non-trivial primitive integer solutions $(a,b,c)$ (respectively, $(a,b,c)$ with $10 | ab$). Finally, in \cite{BCDF17}, Billerey et al. proved that for any fixed regular prime $r \geq 5$, there exists a constant $C_r>0$ such that for primes $p \geq C_r$ the equation $x^{p}+y^{p}=z^{r}$ has non-trivial primitive integer solutions $(a,b,c)$ with $r| ab$ and $2 \nmid ab$.
	
	\subsection*{Generalized Fermat equation of signature $(2p,2q,r)$:}
	In this section, we give the literature for the solutions of the generalized Fermat equation~\eqref{p,q,r} of signature $(2p,2q,r)$ with $r \geq 5$. In \cite{Ben06}, Bennett proved that the equation $x^{2p}+y^{2p}=z^5$ with $p \geq 2$ has no non-trivial primitive integer solutions. In \cite{AS16}, Anni and Siksek proved that the only primitive integer solutions of the equation $x^{2p}+y^{2q}=z^r$ with $r=5,7,11$ (respectively, for $r=13$ with $p,q \neq 7$) are $(\pm1, 0,1)$ and $(0, \pm1, 1)$. Later in \cite{BCDDF23}, Billerey et al. removed the assumption $p,q \neq 7$ for $r=13$ from \cite{AS16} by successfully eliminating the Hilbert modular form occuring in the elimination step i.e., Step 4 of the modular method given in \S\ref{modular method}. In \cite{Mich22}, Michaud-Jacobs proved that the equation $x^{2p}+y^{2q}=z^{17}$ with primes $p,q > 5$ has no non-trivial primitive integer solutions.

	\subsection{Main results}
	\label{notations section for x^r+y^r=dz^p}
	In this subsection, we state the main results of this article.
	More precisely, we prove various results on the effective integer solutions of the generalized Fermat equation~\eqref{2p,2q,r} of signature $(2p,2q,r)$ i.e., $x^{2p}+y^{2q}=z^r$, where the primes $p,q \geq 5$ are varying and the prime $r \geq 5$ is fixed.  
	\begin{dfn}
		A solution $(a, b, c)\in \Z^3$ to the equation \eqref{2p,2q,r} is said to be non-trivial if $abc\neq 0$, and primitive if $a,b,c$ are pairwise coprime. 
	\end{dfn}  
	
	For any number field $F$, let $\mcO_F$ denote its ring of integers and $P_F$ denote the set of all non-zero prime ideals of $\mcO_F$. For any non-zero integer $m$ and any number field $F$, let 
	$$S_{F,m}:=\{\mfp \in P_F: \ \mfp | m \mcO_F\}.$$
	For any set $S \subseteq P_F$, let $\mcO_{S}:=\{\alpha \in F : v_\mfq(\alpha)\geq 0 \text{ for all } \mfq \in P_F \setminus S\}$ be the ring of $S$-integers in $F$ and $\mcO_{S}^*:=\{\alpha \in F : v_\mfq(\alpha)= 0 \text{ for all } \mfq \in P_F \setminus S\}$ be the units of $\mcO_{S}$ which we refer as $S$-units. 
	For any rational prime $r \geq 5$, let $\zeta_r$ denote a primitive $r$th root of unity.	We now state the first main result of this article. More precisely, we study effective integer solutions of the equation~\eqref{2p,2q,r} with $r \nmid x$.
	\begin{thm}[Main result 1]
		\label{main result1 for (2p,2q,r)}
		Fix $r\geq 5$ a rational prime and let $K:= \Q(\zeta_r+ \zeta_r^{-1})$.
		Suppose, for every solution $(\lambda, \mu)$ to the $S_{K,2}$-unit equation
		\begin{equation}
			\label{S_K-unit solution}
			\lambda+\mu=1, \ \lambda, \mu \in \mcO_{S_{K,2}}^\ast,
		\end{equation}
		there exists a prime $\mfP \in S_{K,2}$ that satisfies
		\begin{equation}
			\label{assumption for main result1}
			\max \left\{|v_\mfP(\lambda)|,|v_\mfP(\mu)| \right\}\leq 4v_\mfP(2).
		\end{equation}
		Then there exists an effectively computable constant $V_{r}>0$ (depending on $r$) such that for primes $p,q>V_r$, the equation $x^{2p}+y^{2q}=z^r$ has no non-trivial primitive integer solution $(a,b,c)$ with $r \nmid a$.
	\end{thm}
	We now state the second main result of this article. More precisely, we study the remaining case of the above theorem, i.e., effective integer solutions of the equation~\eqref{2p,2q,r} with $r | x$.
	
	\begin{thm}[Main result 2]
		\label{main result2 for (2p,2q,r)}
		Fix $r\geq 5$ a rational prime and let $K:= \Q(\zeta_r+ \zeta_r^{-1})$.
		Suppose, for every solution $(\lambda, \mu)$ to the $S_{K, 2r}$-unit equation
		\begin{equation}
			\label{S_{K,2r}-unit solution}
			\lambda+\mu=1, \ \lambda, \mu \in \mcO_{S_{K, 2r}}^\ast,
		\end{equation}
		there exists a prime $\mfP \in S_{K,2}$ that satisfies
		\begin{equation}
			\label{assumption for main result2}
			\max \left\{|v_\mfP(\lambda)|,|v_\mfP(\mu)| \right\}\leq 4v_\mfP(2).
		\end{equation}
		Then there exists an effectively computable constant $V_{r}>0$ (depending on $r$) such that for primes $p,q>V_r$, the equation $x^{2p}+y^{2q}=z^r$ has no non-trivial primitive integer solution $(a,b,c)$ with $r |a$.
	\end{thm}
	\begin{remark}
		The main reason for distinguishing two separate cases i.e., $r \nmid a$ in Theorem~\ref{main result1 for (2p,2q,r)} and $r|a$ in Theorem~\ref{main result2 for (2p,2q,r)}, lies in the fact that the associated Frey elliptic curves differ in these two cases (cf. \eqref{Frey curve for r not divide x} and \eqref{Frey curve for r divide x} for Frey curves). Moreover, in Theorem~\ref{main result1 for (2p,2q,r)}, we require a condition on the solutions of the $S_{K, 2}$-unit equation~\eqref{S_K-unit solution} whereas in Theorem~\ref{main result2 for (2p,2q,r)}, we require a condition on the solutions of $S_{K, 2r}$-unit equation~\ref{S_{K,2r}-unit solution}. This distinction arises from the different reduction properties of the Frey curves: specifically, the Frey curve $E/K$ with $r \nmid a$ has good reduction at the unique prime $\mfr$ of $\mcO_K$ lying above $r$, while the Frey curve $E/K$ with $r|a $ has multiplicative reduction at $\mfr$ (cf. Theorem~\ref{semi stable red of Frey curve}).
	\end{remark}
	\begin{remark}
		\label{remark on S unit eqn}
		Note that in Theorem~\ref{main result1 for (2p,2q,r)} (respectively,  in Theorem~\ref{main result2 for (2p,2q,r)}), we deduce a Diophantine problem to a computational problem involving only the solutions of the $S$-unit equation. Recently, significant progress has been made in the study of the solutions of the $S$-unit equation. More precisely, for any number field $F$ and any finite set $S \subseteq P_F$, the $S$-unit equation $\lambda+\mu=1, \ \lambda, \mu \in \mcO_{S}^\ast$ has only a finite number of solutions (cf. \cite{S14} for more details). Furthermore, the solutions of the $S$-unit equation are effectively computable (cf. \cite{AKMRVW21} for more details).
	\end{remark}
	\begin{remark}
		\label{comparision}
		Note that \cite{AS16} (respectively, \cite{Mich22}) studied the integer solutions of the equation $x^{2p}+y^{2q}=z^r$ for $r=3,5,7,11,13$ (respectively, $r=17$). However, for each rational prime $r \geq 5$, we establish a condition on the solutions of the $S$-unit equation over the field $\Q(\zeta_r+ \zeta_r^{-1})$ for which there exists an effectively computable constant $V_{r}>0$ such that the equation $x^{2p}+y^{2q}=z^r$ with primes $p,q>V_r$ has no non-trivial primitive integer solution (cf. Theorems~\ref{main result1 for (2p,2q,r)} and \ref{main result2 for (2p,2q,r)}). By Remark~\ref{remark on S unit eqn}, we can compute all the primes $r$ satisfying the condition on the $S$-unit equation. Moreover, for each rational prime $r \geq 5$, in Propositions~\ref{loc crit1 for main result1} and \ref{loc crit1 for main result2}, we give a conditional criterion for the field  $\Q(\zeta_r+ \zeta_r^{-1})$ satisfying the above condition on the $S$-unit equation and explicitly compute all the primes $r \leq 200$ satisfying Propositions~\ref{loc crit1 for main result1} and \ref{loc crit1 for main result2} (cf. Corollary~\ref{cor1 to loc crit1 for main result2})
	\end{remark}
	
	For any number field $F$, let $h_F^+$ denote the narrow class number of $F$. We now establish the necessary conditions on the prime $r$ for which Theorem~\ref{main result1 for (2p,2q,r)} holds. More precisely,
	\begin{prop}
		\label{loc crit1 for main result1}
		Let $r \geq 5$ be a prime and $K:= \Q(\zeta_r+ \zeta_r^{-1})$. Assume that $2$ is inert in $K$ and $2 \nmid h_K^+$. Then there is an effectively computable constant $V_{r}>0$ (depending on $r$) such that for primes $p,q>V_r$, the equation $x^{2p}+y^{2q}=z^r$ has no non-trivial primitive integer solution $(a,b,c)$ with $r \nmid a$.
	\end{prop}
	
	We now give the necessary conditions on the prime $r$ for which  Theorem~\ref{main result2 for (2p,2q,r)} holds. 
	\begin{prop}
		\label{loc crit1 for main result2}
		Fix $r\geq 5$ a rational prime with $r \not\equiv 1 \pmod 8$ and let $K:= \Q(\zeta_r+ \zeta_r^{-1})$. Assume that $2$ is inert in $K$ and $2 \nmid h_K^+$. Then there is an effectively computable constant $V_{r}>0$ (depending on $r$) such that for primes $p,q>V_r$, the equation $x^{2p}+y^{2q}=z^r$ has no non-trivial primitive integer solution $(a,b,c)$ with $r |a$.
	\end{prop}
	\begin{remark}
		In contrast to Proposition~\ref{loc crit1 for main result1}, we need an extra condition in Proposition~\ref{loc crit1 for main result2}, i.e.,  $r \not\equiv 1 \pmod 8$. This is because of the condition on the solutions $S_{K, 2r}$-unit equation in Theorem~\ref{main result2 for (2p,2q,r)} instead of $S_{K, 2}$-unit equation in Theorem~\ref{main result1 for (2p,2q,r)}.
	\end{remark}
	As an immediate corollary to Propositions~\ref{loc crit1 for main result1} and \ref{loc crit1 for main result2}, we compute all rational primes $r$ with $r \leq 200$ satisfying Proposition~\ref{loc crit1 for main result1},~\ref{loc crit1 for main result2} (cf. \cite{LMFDB25} and \cite{PARI25} for the computations).
	\begin{cor}
		\label{cor1 to loc crit1 for main result2}
		Let $r \in \{5,7,11,13, 19, 23, 37, 47, 53, 59, 61, 67, 71, 79, 83, 101, 103, 107,\\ 131, 139, 149, 163, 167, 173, 179, 181, 191, 197, 199\}$.  
		Then there exists an effectively computable constant $V_{r}>0$ (depending on $r$) such that for primes $p,q>V_r$, the equation $x^{2p}+y^{2q}=z^r$ has no non-trivial primitive integer solution.
	\end{cor}

	\subsection{Methodology:}
	\label{modular method}
	We used the modular method to prove Theorems~\ref{main result1 for (2p,2q,r)} and \ref{main result2 for (2p,2q,r)}. The following are the key steps in the modular method:  
	\begin{steps}
		\item For any non-trivial integer solution to the equation~\eqref{2p,2q,r}, attach a suitable Frey elliptic curve $E$ defined over a totally real number field $K$.
		\item Then prove the modularity, semi-stability of $E/K$ and the mod $p$ Galois representation $\bar{\rho}_{E,p}$ attached to $E$ is irreducible for $p \gg 0$. 
		\item Using level lowering results to conclude that $\bar{\rho}_{E,p} \sim \bar{\rho}_{f,p}$, for some Hilbert modular newform $f$ over $K$ of parallel weight $2$ with rational Hecke eigenvalues and level $\mfn_p$ which is independent of the solution.
		\item Prove that the finitely many Hilbert modular newforms that occur in the above step do not correspond to $\bar{\rho}_{E,p}$, to get a contradiction.
	\end{steps}
	We now explain the proof of our main results, i.e., Theorems~\ref{main result1 for (2p,2q,r)} and ~\ref{main result2 for (2p,2q,r)} briefly in terms of the above steps of the modular method. 
	\begin{itemize}
		\item For any non-trivial integer solution $(a,b,c)$ to the equation $x^{2p}+y^{2q}=z^r$, we attach Frey elliptic curves $E$ defined over the cyclotomic totally real field $K:= \Q(\zeta_r+ \zeta_r^{-1})$  (cf. \eqref{Frey curve for r not divide x}, \eqref{Frey curve for r divide x} for the Frey curves).
		
		\item By Theorem~\ref{semi stable red of Frey curve}, the Frey curve $E/K$ is semistable. Using Theorems~\ref{modularity cycl} and ~\ref{irreducibility of mod $P$ representation}, there exist effectively computable constants $C_r, D_r>0$ such that the Frey curve $E/K$ is modular for $p,q \geq C_r$ and $\bar{\rho}_{E,p}$ is irreducible for $p >D_r$.
		
		\item  Then using the level-loweing result, i.e., Theorem~\ref{level lowering of mod $p$ repr}, we get $\bar{\rho}_{E,p} \sim \bar{\rho}_{f,p}$, for some Hilbert modular newform $f$ of level $\mfn_p$, where $\mfn_p$ as in \eqref{conductor of E for r not divide a} if $r \nmid a$ (respectively, in \eqref{conductor of E for r divide a} if $r |a$).
		
		\item To get a contradiction, as outlined in Step $4$, Theorem~\ref{main result1 for (2p,2q,r)} (respectively, Theorem~\ref{main result2 for (2p,2q,r)}) uses technical ideas from \cite{FS15} which involves the $S_{K,2}$-unit equation \eqref{S_K-unit solution} (respectively, $S_{K,2r}$-unit equation~\eqref{S_{K,2r}-unit solution}). We now explain this step briefly.
		
		First, in Lemma~\ref{reduction on T and S}, we prove that the Frey curve $E/K$ has potential multiplicative reduction at all primes $\mfP \in S_{K,2}$ and $p \nmid v_\mfP(j_E)$ for $p > v_\mfP (2)$. Then by using an image of inertia result, i.e., Lemma~\ref{criteria for potentially multiplicative reduction}, we get  $p | \# \bar{\rho}_{E,p}(I_\mfP)$. Next, by using Theorems~\ref{level lowering of mod $p$ repr} and~\ref{FS partial result of E-S conj}, we prove that that there exists a constant $V_r>0$ (depending on $r$) and an elliptic curve $E'$ defined over $K$ such that 
		$E^\prime/K$ has good reduction away from $S_{K,2}$ if $r \nmid a$ (respectively  $S_{K,2r}$ if $r |a$), $E^\prime/K$ has full $2$-torsion points, $\bar{\rho}_{E,p} \sim \bar{\rho}_{E^\prime,p}$ for primes $p \geq V_r$ (cf. Theorem~\ref{auxilary result for main result1} for more details).
		Then, by Lemma~\ref{criteria for potentially multiplicative reduction}, we conclude that $E'$ has potential multiplicative reduction at all primes $\mfP \in S_{K,2}$, i.e., $v_\mfP(j_{E^\prime}) <0$ for all primes $\mfP \in S_{K,2}$.
		
		For the proof of Theorem~\ref{main result1 for (2p,2q,r)} (respectively, Theorem~\ref{main result2 for (2p,2q,r)}), we first relate the $j$-invariant $j_{E'}$ of $E'$ in terms of the solution $(\lambda, \mu)$ to the $S_{K, 2}$-unit equation~\eqref{S_K-unit solution} (respectively  $S_{K,2r}$-unit equation~\eqref{S_{K,2r}-unit solution}). Finally, to get a contradiction in the proof of Theorem~\ref{main result1 for (2p,2q,r)} (respectively, Theorem~\ref{main result2 for (2p,2q,r)}), we use the explicit bound \eqref{assumption for main result1} (respectively, \eqref{assumption for main result2}) on $\lambda, \mu$ to get $v_\mfP(j_{E^\prime}) \geq0$ for some $\mfP \in S_{K,2}$ (cf. \S\ref{proof of main results} for more details).
	\end{itemize}
	
	
	\subsection{Structure of the article}
	In \S\ref{preliminary section}, we give all the preliminaries to construct the Frey elliptic curves. In \S\ref{theor backgrnd}, we prove all technical results, including the construction of Frey elliptic curves and the modularity of elliptic curves over $\Q(\zeta_r+ \zeta_r^{-1})$ i.e. Theorem~\ref{modularity cycl}, to prove the main results of this article, i.e., Theorems~\ref{main result1 for (2p,2q,r)} and~\ref{main result2 for (2p,2q,r)}. In  \S\ref{proof of main results}, we prove Theorems~\ref{main result1 for (2p,2q,r)} and ~\ref{main result2 for (2p,2q,r)}. In \S\ref{section for local criteria}, we prove Propositions~\ref{loc crit1 for main result1} and \ref{loc crit1 for main result2}.

	\section{Preliminaries}
	\label{preliminary section}
	\subsection{Cyclotomic fields}
	Fix a rational prime $r \geq 5$. In this subsection, we recall some basic preliminaries for the cyclotomic totally real field $K:= \Q(\zeta_r+ \zeta_r^{-1})$. Clearly $[K: \Q]= \frac{r-1}{2}$, $r$ is totally ramified in $K$ and $r\mcO_K= \mfr^{ \frac{r-1}{2}}$, where $\mfr$ is the unique prime ideal of $\mcO_K$ lying above $r$. 
	We fix the notation
	$$\theta_j:= \zeta_r^j+ \zeta_r^{-j}, \text{ for any } j \in \{1, \dots, \frac{r-1}{2}\}.$$
	The following lemma will be very useful in the sequel.
	\begin{lem}{\cite[Lemma 3.1]{AS16}}
		\label{AS cycl lem}
		Let $K= \Q(\zeta_r+ \zeta_r^{-1})$. For any $j \in \{1, \dots, \frac{r-1}{2}\}$, we have 
		$$\theta_j \in \mcO_K^\ast, \ \theta_j+2 \in \mcO_K^\ast \text{ and }   (\theta_j-2)\mcO_K=\mfr.$$
		Further, for $1 \leq j <k \leq \frac{r-1}{2}$, we have 
		$$ (\theta_j-\theta_k)\mcO_K=\mfr.$$
	\end{lem}
	\subsection{The descent}
	In this section, we recall the descent argument over the field $K= \Q(\zeta_r+ \zeta_r^{-1})$, which will be very useful for the construction of the Frey curves defined over $K$ (cf. \cite[\S4]{AS16} for more details).
	
	Let $(a,b,c) $ be a non-trivial primitive integer solution of the equation $x^{2p}+y^{2q}=z^r$. Since $(a,b,c)$ is primitive and $a^{2p}+b^{2q}=c^r$, it follows that $c$ is odd, and hence  one of $a,b$ is even. Without loss of generality, we can assume $a$ is even, hence $b$ is odd. Since $a^p$ and $b^q$ are coprime integers and $c$ is odd, it follows that $a^p+b^qi$ and $a^p-b^qi$ are coprime in $\Z[i]$. As $a^{2p}+b^{2q}=c^r$, there exists coprime integers $a_0$, $b_0$ such that $a^p+b^qi= (a_0+b_0i)^r$ and $c=a_0^2+b_0^2$. This gives 
	$$a^p =\frac{1}{2} \left((a_0+b_0i)^r+(a_0-b_0i)^r \right).$$
	For any $u,v \in \C$ and odd prime $r$, we recall the standard factorization
	$$u^r+v^r= \prod_{j=0}^{r-1} (u+v\zeta_r^j).$$ This gives 
	\begin{align*}
		& a^p =a_0 \prod_{j=1}^{r-1} \left((a_0+b_0i)+(a_0-b_0i)\zeta_r^j \right) \\
		&=a_0 \prod_{j=1}^{\frac{r-1}{2}} \left((a_0+b_0i)+(a_0-b_0i)\zeta_r^j \right) \left((a_0+b_0i)+(a_0-b_0i)\zeta_r^{-j} \right) \\
		&= a_0 \prod_{j=1}^{\frac{r-1}{2}} \left((\theta_j+2)a_0^2+(\theta_j-2)b_0^2 \right).
	\end{align*}
	Hence 
	\begin{equation}
		a^p= a_0 \prod_{j=1}^{\frac{r-1}{2}} \beta_j, 
	\end{equation}
	where 
	\begin{equation}
		\label{beta_j}
		\beta_j= \left((\theta_j+2)a_0^2+(\theta_j-2)b_0^2 \right), \text{ for } j \in \{1, \dots \frac{r-1}{2} \}.
	\end{equation}
	Hence $\beta_j \in \mcO_K$ for all $j \in \{1, \dots \frac{r-1}{2} \}$. 
	The following lemma give the factorization of $a_0$ and  $\beta_j \mcO_K$, which will be very useful for the computations of the conductor of the Frey elliptic curves over $K$.
	\begin{lem}{\cite[Lemma 4.1]{AS16}}
		\label{AS Lem}
		Let $(a,b,c) $ be a non-trivial primitive integer solution to equation $x^{2p}+y^{2q}=z^r$ and let $n=v_2(a) $. Let $\mfr$ be the unique prime ideal of $\mcO_K$ lying above $r$.
		\begin{enumerate}
			\item If $r \nmid a$, then $a_0= 2^{pn}a_1^p$ and $\beta_j \mcO_K= \mfb_j^p$, where $a_1$ is a rational integer and $a_1 \mcO_K$, $\mfb_1, \dots , \mfb_\frac{r-1}{2}$ are pairwise coprime ideals in $\mcO_K$ which are coprime to $2r$.
			\item If $r | a$, then $a_0= 2^{pn}r^{kp-1}a_1^p$ and $\beta_j \mcO_K= \mfr \mfb_j^p$, where $k=v_r(a)$, $a_1$ is a rational integer and $a_1 \mcO_K$, $\mfb_1, \dots , \mfb_\frac{r-1}{2}$ are pairwise coprime ideals in $\mcO_K$ which are coprime to $2r$.
		\end{enumerate}
	\end{lem}
	
	\section{Theoretical background}
	\label{theor backgrnd}
	\subsection{Construction of Frey elliptic curve}
	\label{section for Frey curve}
	In this subsection, we recall the Frey elliptic curve associated to any non-trivial primitive integer solution to the equation $x^{2p}+y^{2q}=z^r$ (cf. \cite[\S 7]{AS16} for more details). Fix a rational prime $r \geq 5$ and fix any two distinct integers $j,k \in \{1, \dots, \frac{r-1}{2}\}$. Recall that $K= \Q(\zeta_r+ \zeta_r^{-1})$. 
	
	Let $(a,b,c) \in \Z^3$ be a non-trivial primitive solution of the equation \eqref{2p,2q,r} and let $\beta_j, \beta_k$ as in \eqref{beta_j}. The idea is that the three terms $\beta_j, \beta_k, a_0^2$ are roughly $p$-th power (cf. Lemma~\ref{AS Lem}) and they are linearly dependent over $K$. So we can find $\alpha, \beta, \gamma \in \mcO_K \setminus \{0\}$ such that $\alpha\beta_j + \beta \beta_k+ \gamma a_0^2=0$, which gives a Fermat equation of signature $(p,p,p)$ over $K$ and hence we can associate a Frey curve over $K$.
	More precisely;
	\begin{enumerate}
		\item If $r \nmid a$, then choose 
		\begin{equation}
			\label{alpha, beta, gamma, r not divide a}
			\alpha=1, \ \beta= \frac{\theta_j-2}{\theta_k-2},\ \gamma=4\frac{\theta_j-\theta_k}{\theta_k-2}.
		\end{equation} 
		Clearly, $\alpha, \beta, \gamma \in \mcO_K \setminus \{0\}$ and 
		$\alpha\beta_j + \beta \beta_k+ \gamma a_0^2=0$. Moreover by Lemmas~\ref{AS cycl lem} and~\ref{AS Lem}, we have 
		\begin{equation}
			\label{beta_j, bet_k, gamma, r not divide a}
			\beta_j \mcO_K= \mfb_j^p, \ \beta \beta_k \mcO_K= \mfb_k^p, \  \gamma a_0^2\mcO_K= 2^{2pn+2}a_1^{2p}\mcO_K,
		\end{equation}
		where $n=v_2(a)$. Now, consider the Frey elliptic curve $E/K$ as follows:
		\begin{equation}
			\label{Frey curve for r not divide x}
			E:=E_{a,b,c}: Y^2=X(X-\beta_j) (X+ \beta  \beta_k ).
		\end{equation}
		Here, $c_4=2^4(\beta_j^2-\beta \gamma \beta_k a_0^2)$,  $\Delta_E=2^4(\beta\gamma \beta_j \beta_ka_0^2)^2$ and $j_E= \frac{c_4^3}{\Delta_E}$.
		\item If $r| a$, then choose 
		\begin{equation}
			\label{alpha, beta, gamma, r divide a}
			\alpha=\frac{1}{\theta_j-2}, \ \beta= - \frac{1}{\theta_k-2},\ \gamma=4\frac{\theta_j-\theta_k}{(\theta_j-2)(\theta_k-2)}.
		\end{equation} 
		Clearly, $\alpha, \beta, \gamma \in \mcO_K \setminus \{0\}$ and 
		$\alpha\beta_j + \beta \beta_k+ \gamma a_0^2=0$. Moreover by Lemmas~\ref{AS cycl lem} and ~\ref{AS Lem}, we get 
		\begin{equation}
			\label{beta_j, bet_k, gamma, r divide a}
			\alpha\beta_j \mcO_K= \mfb_j^p, \ \beta \beta_k \mcO_K= \mfb_k^p, \  \gamma a_0^2\mcO_K= 2^{2pn+2} \mfr^\delta a_1^{2p}\mcO_K, 
		\end{equation}
		where $\delta= (kp-1)(r-1)-1$.	
		Consider the Frey elliptic curve $E/K$ as follows:
		\begin{equation}
			\label{Frey curve for r divide x}
			E:=E_{a,b,c}: Y^2=X(X-\alpha\beta_j) (X+ \beta  \beta_k )
		\end{equation}
		Here, $c_4=2^4(\alpha^2 \beta_j^2-\beta \gamma \beta_k a_0^2)$,  $\Delta_E=2^4(\alpha \beta\gamma \beta_j \beta_ka_0^2)^2$ and $j_E= \frac{c_4^3}{\Delta_E}$.
	\end{enumerate}
	
	\subsection{Conductor of the Frey elliptic curve}
	In this subsection, we study the reduction type of the Frey elliptic curves $E/K$ given in \eqref{Frey curve for r not divide x} and \eqref{Frey curve for r divide x}.
	For any non-zero prime ideal $\mfq$ of $\mcO_K$, let $\Delta_\mfq$ be the minimal discriminant of $E$ at $\mfq$. For any rational prime $p$, let 
	\begin{equation}
		\label{conductor of elliptic curve}
		\mfm_p:= \prod_{ p|v_\mfq(\Delta_\mfq) \text{ and}\ \mfq ||\mfn} \mfq \text{ and } \mfn_p:=\frac{\mfn}{\mfm_p}.
	\end{equation}
	The following theorem determines the conductor of the Frey elliptic curves $E/K$ given in \eqref{Frey curve for r not divide x} and \eqref{Frey curve for r divide x}.
	
	\begin{thm}
		\label{semi stable red of Frey curve}
		Let $(a,b,c) \in \Z^3$ be a non-trivial primitive solution of the equation \eqref{2p,2q,r} and let $E/K$ be the Frey curve given in \eqref{Frey curve for r not divide x} for $r \nmid a$ (respectively in \eqref{Frey curve for r divide x} for $r|a$). Let $\mfr$ be the unique prime ideal of $\mcO_K$ lying above $r$.
		\begin{enumerate}
			\item If $r \nmid a$, then $E/K$ is semistable, good reduction at $\mfr$ and multiplicative reduction at all primes of $\mcO_K$ lying above $2$. Moreover, $E$ is minimal at all primes $\mfq$ of $\mcO_K$ with $\mfq \nmid 2$ and satisfies $p | v_\mfq(\Delta_E)$. Let $\mfn$ be the conductor of $E/ K^+$ and $\mfn_p$ be as in \eqref{conductor of elliptic curve}. Then,
			\begin{equation}
				\label{conductor of E for r not divide a}
				\mfn=2\mcO_K\prod_{\mfq \in P_K,\ \mfq |a_0\beta_j \beta_K} \mfq,\ \mfn_p=2\mcO_K.
			\end{equation}
			
			\item If $r| a$, then $E/K$ is semistable, multiplicative reduction at $\mfr$ and all primes of $\mcO_K$ lying above $2$. Moreover, $E$ is minimal at all primes $\mfq$ of $\mcO_K$ with $\mfq \nmid 2r$ and and satisfies $p | v_\mfq(\Delta_E)$. Then,
			\begin{equation}
				\label{conductor of E for r divide a}
				\mfn=2\mcO_K \mfr \prod_{\mfq \in P_K,\ \mfq |a_0\beta_j \beta_K} \mfq,\ \mfn_p=2\mcO_K\mfr.
			\end{equation}
		\end{enumerate}
	\end{thm}
	
	\begin{proof}
		\begin{enumerate}
			\item If $r \nmid a$, then by ~\cite[Lemma 6.1]{AS16}, $E/K$ is semistable, good reduction at $\mfr$, multiplicative reduction at all primes of $\mcO_K$ lying above $2$ and the conductor $\mfn$ of $E$ is given by \eqref{conductor of E for r not divide a}.
			Recall that here $c_4=2^4(\beta_j^2-\beta \gamma \beta_k a_0^2)$ and $\Delta_E=2^4(\beta\gamma \beta_j \beta_ka_0^2)^2$.
			If $\mfq \nmid \Delta_E$, then $E/K$ has good reduction at $\mfq$ and hence $p | v_\mfq(\Delta_E)=0$. 
			
			If $\mfq | \Delta_E$, then $\mfq | \beta_j \beta  \beta_k \gamma a_0^2$ because $\mfq \nmid 2$. By equation~\eqref{beta_j, bet_k, gamma, r not divide a}, we have $\beta_j \mcO_K= \mfb_j^p$, $\beta \beta_k \mcO_K= \mfb_k^p$ and $\gamma a_0^2\mcO_K= 2^{2pn+2}a_1^{2p}\mcO_K$. Since $\mfq | \beta_j \beta  \beta_k \gamma a_0^2$ and by Lemma~\ref{AS Lem}(1), it follows that $\mfq$ divides exactly one of $\beta_j$,  $\beta  \beta_k$, $ \gamma a_0^2$ because $\mfq \nmid 2$. 
			This gives $v_\mfq (c_4)=0$, hence $E$ is minimal and has multiplicative reduction at $\mfq$. 
			Since $\mfq \nmid 2$, $\beta_j \mcO_K= \mfb_j^p$, $\beta \beta_k \mcO_K= \mfb_k^p$ and $\gamma a_0^2\mcO_K= 2^{2pn+2}a_1^{2p}\mcO_K$, we have $p|v_\mfq(\Delta_E)=2p\left(v_\mfq(\mfb_j)+v_\mfq(\mfb_k)+v_\mfq(a_1^2) \right)$. Finally, using the definition of $\mfn_p$ in~\eqref{conductor of elliptic curve}, we have $\mfn_p=2\mcO_K$.

			\item If $r |a$, then by ~\cite[Lemma 6.2]{AS16}, $E/K$ is semistable, multiplicative reduction at the prime $\mfr$ and all primes of $\mcO_K$ lying above $2$, and the conductor $\mfn$ of $E$ is given by \eqref{conductor of E for r divide a}.
			Recall that here $c_4=2^4(\alpha^2\beta_j^2-\beta \gamma \beta_k a_0^2)$ and $\Delta_E=2^4(\alpha \beta\gamma \beta_j \beta_ka_0^2)^2$.
			If $\mfq \nmid \Delta_E$, then $E/K$ has good reduction at $\mfq$ and hence $p | v_\mfq(\Delta_E)=0$. 
			
			If $\mfq | \Delta_E$, then $\mfq | \alpha \beta_j \beta  \beta_k \gamma a_0^2$ because $\mfq \nmid 2$. By equation~\eqref{beta_j, bet_k, gamma, r divide a}, we have $\alpha \beta_j \mcO_K= \mfb_j^p$, $\beta \beta_k \mcO_K= \mfb_k^p$ and $\gamma a_0^2\mcO_K= 2^{2pn+2}\mfr^\delta a_1^{2p}\mcO_K$. Since  $\mfq | \alpha \beta_j \beta  \beta_k \gamma a_0^2$ and by Lemma~\ref{AS Lem}(2), it follows that $\mfq$ divides exactly one of $\alpha \beta_j$,  $\beta  \beta_k$, $ \gamma a_0^2$ because $\mfq \nmid 2r$. 
			This gives $v_\mfq (c_4)=0$, hence $E$ is minimal and has multiplicative reduction at $\mfq$. 
			Since $\mfq \nmid 2r$, $\alpha \beta_j \mcO_K= \mfb_j^p$, $\beta \beta_k \mcO_K= \mfb_k^p$ and $\gamma a_0^2\mcO_K= 2^{2pn+2}\mfr^\delta a_1^{2p}\mcO_K$, we have $p|v_\mfq(\Delta_E)=2p\left(v_\mfq(\mfb_j)+v_\mfq(\mfb_k)+v_\mfq(a_1^2) \right)$. Now, using the definition of $\mfn_p$ in~\eqref{conductor of elliptic curve}, we have $\mfn_p=2\mcO_K \mfr$.
		\end{enumerate} 
	\end{proof}
	
	\subsection{Modularity of the Frey elliptic curve}
	In this subsection, we discuss the modularity of the Frey curves $E/ K$  given in \eqref{Frey curve for r not divide x} and \eqref{Frey curve for r divide x}. First, we recall the definition of the modularity of elliptic curves over totally real number fields.
	\begin{dfn}
		Let $F$ be a totally real number field. We say an elliptic curve $E/F$ is modular if there exists a primitive Hilbert modular newform $f$ over $F$ of parallel weight $2$ with rational eigenvalues such that both $E$ and $f$ have the same $L$-function over $F$.
	\end{dfn} 
	We first recall a modularity result of Freitas, Le Hung, and Siksek for totally real fields. 
	\begin{thm} \rm(\cite[Theorem 5]{FLHS15})
		\label{modularity result of elliptic curve over totally real}
		Let $F$ be a totally real number field. Then, up to isomorphism over $\bar{F}$, there are only finitely many elliptic curves over $F$ which are not modular. 
	\end{thm}
	The following lemma is a consequence of Theorem~\ref{modularity result of elliptic curve over totally real}, which proves the modularity of the Frey curves $E$ in \eqref{Frey curve for r not divide x} and \eqref{Frey curve for r divide x} for large primes $p$ and $q$.
	\begin{lem}{\cite[Lemma 8.1]{AS16}}
		\label{modularity result for main result1}
		Fix $r\geq 5$ a rational prime and let $K:= \Q(\zeta_r+ \zeta_r^{-1})$.  Then, there exists a constant $C_r>0$ (depending on $r$) such that for any non-trivial primitive integer solution $(a,b,c)$ to the equation~\eqref{2p,2q,r} with primes $p, q \geq C_r$, the Frey elliptic curve $E/{K}$ given in \eqref{Frey curve for r not divide x} for $r \nmid a$ (respectively, in \eqref{Frey curve for r divide x} for $r|a$) is modular.  
	\end{lem}
	
	Note that the constant $C_r$ in the above lemma is ineffective. We now prove that every elliptic curve over the cyclotomic totally real field $K:= \Q(\zeta_r+ \zeta_r^{-1})$ is modular, which makes $C_r$ effective and in fact we can take $C_r=5$ for all $r\geq 5$. This will be very helpful to show the constant $V_r$ in Theorem~\ref{main result1 for (2p,2q,r)} (respectively, in Theorem \ref{main result2 for (2p,2q,r)}) is effectively computable (cf. \S\ref{proof of main results} for more details).
	\begin{thm}
		\label{modularity cycl}
		Let $r\geq 2$ be a rational prime. Then every elliptic curve over $\Q(\zeta_r+ \zeta_r^{-1})$ is modular. In particular, the Frey curves $E$ in \eqref{Frey curve for r not divide x} and \eqref{Frey curve for r divide x} are modular for all primes $p, q \geq 5$.
	\end{thm}
	
	\begin{proof}
		For any rational prime $r \geq 2$, let $K:=\Q(\zeta_r+ \zeta_r^{-1})$.
		Let $E$ be an elliptic curve defined over $K$.
		For $r=2,3$, we get $K=\Q$ and hence $E$ is modular over $K$ by \cite{TW95}, \cite{W95} and \cite{BCDT01}. For $r=5$, we get $[K:\Q]=2$ and hence $E$ is modular over $K$ by  \cite[Theorem 1]{FLHS15}. For $r=7$, we get $[K:\Q]=3$ and hence $E$ is modular over $K$ by  \cite[Theorem 4]{DNS20}. 
		For $r\geq 11$, the field $K:= \Q(\zeta_r+ \zeta_r^{-1})$ is a finite abelian extension of $\Q$ which is unramified at primes away $r$ and hence $K$ is unramified at primes $3,5,7$. Thus by \cite[Theorem 1.2]{Y18}, $E$ is modular over $K$.
	\end{proof} 
	\subsection{Irreducibility of the mod $p$ Galois representations}
	Let $F$ be a number field and let $E/F$ be an elliptic curve defined over $F$. For any rational prime $p \geq 2$, let $E[p]:=\{P \in E(\bar{K}):\ [p]P=O\}$ be the set of all $p$-torsion points of $E$. Let
	$$\bar{\rho}_{E,p} : G_F:=\Gal(\bar{F}/F) \rightarrow \mathrm{Aut}(E[p]) \simeq \GL_2(\F_p)$$
	be the mod $p$ Galois representation of the absolute Galois group $G_F$, induced by the action of $G_F$ on $E[p]$.
	In this subsection, we discuss the irreducibility of the mod $p$ Galois representations $\bar{\rho}_{E,p}$ for large primes $p$. The following theorem gives a criterion for determining the irreducibility of mod $p$ Galois representations $\bar{\rho}_{E,p}$. More precisely,
	
	\begin{thm} \rm(\cite[Theorem 2]{FS15 Irred})
		\label{irreducibility of mod $P$ representation}
		Let $F$ be a totally real Galois field. Then there exists an effective constant $D_F>0$ (depending on $F$) such that if $p>D_F$ is a prime and $E/F$ is an elliptic curve over $F$ which is semistable at all primes $\mfp$ of $F$ with $\mfp |p$, then $\bar{\rho}_{E,p}$ is irreducible.
	\end{thm}	
	
	\subsection{Level lowering}
	Let $f$ be any Hilbert modular newform defined over a totally real number field $F$ of parallel weight $2$, level $\mfn$ with coefficient field $\Q_f$.
	For any non-zero prime ideal $\lambda$ of $\mcO_{\Q_f}$, let $\bar{\rho}_{f, \lambda}: G_F \rightarrow \GL_2(\F_\lambda)$ be the residual Galois representation attached to $f, \lambda$.
	We now recall a standard level-lowering result of Freitas and Siksek from ~\cite{FS15}, which follows from Jarvis \cite{J04}, Fujiwara \cite{F06} and Rajaei \cite{R01}.
	\begin{thm} \rm(\cite[Theorem 7]{FS15})
		\label{level lowering of mod $p$ repr}
		Let $E$ be an elliptic curve defined over a totally real number field $F$ of conductor $\mfn$. Let $p\geq 2$ be a rational prime and $\mfn_p$ as in \eqref{conductor of elliptic curve}. Suppose that the following conditions hold:
		\begin{enumerate}
			\item  For $p \geq 5$, the ramification index $e(\mfp /p) < p-1$ for all prime $\mfp |p$, and $\Q(\zeta_p)^+ \nsubseteq F$;
			\item $E/F$ is modular;
			\item $\bar{\rho}_{E,p}$ is irreducible;
			\item $E$ is semistable at all $\mfp |p$;
			\item $p| v_\mfp(\Delta_\mfp)$ for all $\mfp |p$.
		\end{enumerate}
		Then there exists a Hilbert modular newform $f$ defined over $F$ of parallel weight $2$, level $\mfn_p$, and some prime ideal $\lambda|p$ of $\mcO_{\Q_f}$ such that $\bar{\rho}_{E,p} \sim \bar{\rho}_{f,\lambda}$.
	\end{thm}	
	
	\subsection{Eichler-Shimura}
	We now state the Eichler-Shimura conjecture.
	\begin{conj}[Eichler-Shimura]
		\label{ES conj}
		Let $F$ be a totally real number field. Let $f$ be a Hilbert modular newform defined over $F$ of parallel weight $2$, level $\mfn$, and with coefficient field $\Q_f= \Q$. Then, there exists an elliptic curve $E_f /F$ with conductor $\mfn$ having the same $L$-function as $f$ over $F$.
	\end{conj}
	The above conjecture is true over all totally real number fields $F$ with either 
	$[F: \Q] $ is odd or there exists some prime ideal $\mfq$ of $\mcO_F$ such that $v_\mfq(\mfn) = 1$ (cf. ~\cite[Theorem 7.7]{D04}).
	In \cite{FS15}, Freitas and Siksek provided a partial answer to Conjecture~\ref{ES conj} in terms of mod $p$ Galois representations $\bar{\rho}_{E,p}$. More precisely,
	\begin{thm} \rm(\cite[Corollary 2.2]{FS15})
		\label{FS partial result of E-S conj}
		Let $E$ be an elliptic curve defined over a totally real number field $F$ and $p$ be an odd prime.
		Suppose that $\bar{\rho}_{E,p}$ is irreducible and $\bar{\rho}_{E,p} \sim \bar{\rho}_{f,p}$ for some Hilbert modular newform $f$ defined over $F$ of parallel weight $2$ and level
		$\mfn$ with rational eigenvalues.
		Let $\mfq \in P_F$ with $\mfq \nmid p$ be such that
		\begin{enumerate}
			\item E has potential multiplicative reduction at $\mfq$ (i.e., $v_\mfq(j_E) <0$);
			\item $p| \# \bar{\rho}_{E,p}(I_\mfq)$;
			\item  $p \nmid \left(\mathrm{Norm}(F/\Q)(\mfq) \pm 1\right)$.
		\end{enumerate}
		Then there exists an elliptic curve $E_f /F$ of conductor $\mfn$ having the same $L$-function as $f$.
	\end{thm}
	
	\subsection{Image of inertia}
	The following result is very useful for determining the reduction types of the Frey curve $E$ at the primes of $\mcO_K$ lying above $2$. 
	\begin{lem}\cite[Lemma 3.4]{FS15}
		\label{criteria for potentially multiplicative reduction}
		Let $F$ be a totally real number field and $p\geq 5$ be a rational prime. Let $E$ be an elliptic curve defined over $F$. For $\mfq \in P_F$ with $\mfq \nmid p$, $E$ has potential multiplicative reduction at $\mfq$ (i.e., $v_\mfq(j_E) <0$) and $p \nmid v_\mfq(j_E)$ if and only if $p | \# \bar{\rho}_{E,p}(I_\mfq)$.
	\end{lem}
	
	
	\begin{lem}
		\label{reduction on T and S}
		Let $\mfP \in S_{K,2}$. Let $(a,b,c)$ be a non-trivial primitive integer solution of the equation \eqref{2p,2q,r} with $p > \max\{ 2, v_\mfP(2) \}$, and let $E/K$ be the Frey curve given in \eqref{Frey curve for r not divide x} for $r \nmid a$ (respectively, in \eqref{Frey curve for r divide x} for $r|a$). Then $\ v_\mfP(j_E) < 0$ and $p \nmid v_\mfP(j_E)$, equivalently $p | \#\bar{\rho}_{E,p}(I_\mfP)$.
	\end{lem}
	
	\begin{proof}
		By Theorem~\ref{semi stable red of Frey curve}, the Frey curve $E/K$ has multiplicative reduction at $\mfP$ and hence $v_\mfP(j_E) <0$.
		
		If $r \nmid a$, then in this case we have $j_E=2^8 \frac{(\beta_j^2-\beta \beta_k \gamma  a_0^2)^3}{(\beta_j \beta \beta_k \gamma a_0^2)^2}$. Recall that by equation~\eqref{beta_j, bet_k, gamma, r not divide a}, we have $\beta_j \mcO_K= \mfb_j^p$, $\beta \beta_k \mcO_K= \mfb_k^p$ and $\gamma a_0^2\mcO_K= 2^{2pn+2}a_1^{2p}\mcO_K$, where $n=v_2(a)$. 
		By Lemma~\ref{AS Lem}(1), it follows that $a_1 \mcO_K$, $\mfb_j, \mfb_k$ are pairwise coprime ideals in $\mcO_K$ which are coprime to $2r$. Since $\mfP \in S_{K,2}$, we have $\mfP \nmid \mfb_j \mfb_k a_1 \mcO_K$. Hence $v_\mfP(j_E)=8 v_\mfP(2)-2 v_\mfP(\gamma a_0^2)=8 v_\mfP(2)-2 (2pn+2) v_\mfP(2)=4(1-pn)v_\mfP(2)$. Since $p > \max\{ 2, v_\mfP(2) \}$, we get $p \nmid v_\mfP(j_E)$.
		
		If $r | a$, then in this case we have $j_E=2^8 \frac{(\alpha^2 \beta_j^2-\beta \beta_k \gamma a_0^2)^3}{(\alpha \beta_j \beta \beta_k \gamma a_0^2)^2}$. Recall that by equation~\eqref{beta_j, bet_k, gamma, r divide a}, we have
		$\alpha \beta_j \mcO_K= \mfb_j^p$, $\beta \beta_k \mcO_K= \mfb_k^p$ and  $\gamma a_0^2\mcO_K= 2^{2pn+2}\mfr^\delta a_1^{2p}\mcO_K$, where $\delta= (kp-1)(r-1)-1$. By Lemma~\ref{AS Lem}(2), it follows that $a_1 \mcO_K$, $\mfb_j, \mfb_k$ are pairwise coprime ideals in $\mcO_K$ which are coprime to $2r$. 
		Hence $v_\mfP(j_E)=8 v_\mfP(2)-2 v_\mfP(\gamma a_0^2) =8 v_\mfP(2)-2 (2pn+2) v_\mfP(2)=4(1-pn)v_\mfP(2)$. Since $p > \max\{ 2, v_\mfP(2) \}$, we get $p \nmid v_\mfP(j_E)$.
		
		Finally by Lemma~\ref{criteria for potentially multiplicative reduction}, we have $p | \#\bar{\rho}_{E,p}(I_\mfP)$.
	\end{proof}
	
	\section{Level-lowering and Eicher-Shimura}
	\label{proof of main results}
	In this section, we will prove Theorems~\ref{main result1 for (2p,2q,r)} and ~\ref{main result2 for (2p,2q,r)}. 
	First, we prove an auxiliary theorem which is a key ingredient in the proof of Theorems~\ref{main result1 for (2p,2q,r)} and~\ref{main result2 for (2p,2q,r)}.
	\begin{thm}
		\label{auxilary result for main result1}
		Fix $r\geq 5$ a rational prime. Let $K:= \Q(\zeta_r+ \zeta_r^{-1})$. Then, there is an effectively computable constant $V_{r}>0$ (depending on $r$) such that the following holds. Let $(a,b,c)$ be a non-trivial primitive integer solution of the equation $x^{2p}+y^{2q}=z^r$ with primes $p,q > V_r$, and let $E/K$ be the Frey curve given in \eqref{Frey curve for r not divide x} for $r \nmid a$ (respectively, in \eqref{Frey curve for r divide x} for $r|a$). Then, there exists an elliptic curve $E^\prime/K$ such that:
		\begin{enumerate}
			\item $E^\prime/K$ has good reduction away from $S_{K,2}$ if $r \nmid a$ (respectively,  $S_{K,2r}$ if $r |a$); 
			\item $E^\prime/K$ has full $2$-torsion points i.e. $E'(\bar{K})[2] \subset E'(K)$;
			\item $\bar{\rho}_{E,p} \sim \bar{\rho}_{E^\prime,p}$;
			\item  $v_\mfP(j_{E^\prime})<0$ for all $\mfP \in S_{K,2}$.
		\end{enumerate}
	\end{thm}
	\begin{proof}
		The main idea to prove this theorem is to apply the level-lowering theorem, i.e., Theorem~\ref{level lowering of mod $p$ repr}, to the Frey elliptic curve $E$ for large primes $p,q$ and the Eichler-Shimura result, i.e., Theorem~\ref{FS partial result of E-S conj}, to get such an elliptic curve $E'$ with desired properties. We now proceed step by step.
		
		First, taking $p > r$, we have $e(\mfp /p) < p-1$ for all prime ideals $\mfp$ of $\mcO_K$ with $\mfp |p$, and $\Q(\zeta_p)^+ \nsubseteq K$. By Theorem~\ref{modularity cycl}, the Frey curve $E$ is modular over $K$ for primes $p,q \geq 5$. By Theorem~\ref{semi stable red of Frey curve} with $p >2r$, the Frey curve $E/K$ is semistable at $\mfp$ and $p| v_\mfp(\Delta_\mfp)$ for all prime ideals $\mfp$ of $\mcO_K$ with $\mfp |p$. Since $K= \Q(\zeta_r+ \zeta_r^{-1})$ is totally real and Galois extension of $\Q$, by Theorem~\ref{irreducibility of mod $P$ representation}, we get $\bar{\rho}_{E,p}$ is irreducible for $p > D_r$, for some effective constant $D_r>0$ depending on $r$.
		
		Now, applying Theorem~\ref{level lowering of mod $p$ repr} to the Frey curve $E$ for primes $p,q > V_r:=\max\{2r, D_r\}$, there exists a Hilbert modular newform $f$ defined over $K$ of parallel weight $2$, level $\mfn_p$ and some prime $\lambda$ of $\Q_f$ such that $\lambda | p$ and $\bar{\rho}_{E,p} \sim \bar{\rho}_{f,\lambda}$. 
		After possibly enlarging $V_r$ by an effective amount, we can assume $\Q_f=\Q$. This step uses standard ideas, originally due to Mazur, which can be found in~\cite[Proposition 15.4.2]{C07} (cf.~\cite[\S 4]{FS15} for more details).
		
		Let $\mfP \in S_{K,2}$. By Lemma~\ref{reduction on T and S}, $E$ has potential multiplicative reduction at $\mfP$ i.e.,  $\ v_\mfP(j_E) < 0$ and $p | \#\bar{\rho}_{E,p}(I_\mfP)$
		for $p > \max\{ 2, v_\mfP(2) \}$. Now, by Theorem~\ref{FS partial result of E-S conj}, there exists an elliptic curve $E_f$ of conductor $\mfn_p$ such that $\bar{\rho}_{E,p} \sim \bar{\rho}_{E_f,p}$ after possibly enlarging $V_r$ by an effective amount to ensure that $p \nmid \left( \text{Norm}(K/\Q)(\mfP) \pm 1 \right)$ and $p > v_\mfP(2)$. Then $\bar{\rho}_{E,p} \sim \bar{\rho}_{E_f,p}$ for $p>V_r$.
		
		Since the conductor of $E_f$ is $\mfn_p$ given in \eqref{conductor of E for r not divide a} if $r \nmid a$ (respectively, in \eqref{conductor of E for r divide a} if $r|a$), $E_f$ has good reduction away from $S_{K,2}$ if $r \nmid a$ (respectively,  $S_{K,2r}$ if $r |a$). Clearly from the construction of the Frey curves in \eqref{Frey curve for r not divide x} and \eqref{Frey curve for r divide x}, the Frey curve $E/K$ has full $2$-torsion points given by  $\{O, (0,0), (0, \beta_j), (0,\beta \beta_k)\}$, where $\beta \in K$ as in \eqref{alpha, beta, gamma, r not divide a} and $\beta_j, \beta_k \in K$ as in \eqref{beta_j} if $r \nmid a$ (respectively, $\{O, (0,0), (0, \alpha \beta_j), (0,\beta \beta_k)\}$, where $\alpha, \beta \in K$ as in \eqref{alpha, beta, gamma, r divide a} if $r | a$).
		After enlarging $V_r$ by an effective amount and by possibly replacing $E_f$ with an isogenous curve, say $E^\prime$, we may assume that $E^\prime/ K$ has full $2$-torsion and  $\bar{\rho}_{E,p} \sim\bar{\rho}_{E^\prime,p}$. 
		This follows from~\cite[Proposition 15.4.2]{C07} and the fact that  $E/K$ has full $2$-torsion points (cf.~\cite[\S 4]{FS15} for more details). Finally, since $E_f$ is isogenous to $E^\prime$, it follows that $E^\prime$ has good reduction away from $S_{K,2}$ if $r \nmid a$ (respectively,  $S_{K,2r}$ if $r |a$).
		
		Now by Lemma~\ref{reduction on T and S}, we have $p |\# \bar{\rho}_{E,p}(I_\mfP)= \# \bar{\rho}_{E^\prime,p}(I_\mfP)$ for all $\mfP \in S_{K,2}$. So by Lemma~\ref{criteria for potentially multiplicative reduction}, we get $v_\mfP(j_{E^\prime})<0$.
		This completes the proof of the theorem.
	\end{proof}

	We are now in a position to complete the proof of Theorem~\ref{main result1 for (2p,2q,r)} by applying Theorem~\ref{auxilary result for main result1}. The proof of this theorem is analogous to that of \cite[Theorem 3]{FS15} and \cite[Theorem 2.2]{JS25}.
	\begin{proof}[Proof of Theorem~\ref{main result1 for (2p,2q,r)}]
		We will prove this theorem by contradiction. Let $r\geq 5$ be a rational prime and let $K:= \Q(\zeta_r+ \zeta_r^{-1})$. Suppose
		$(a,b,c)$ is a non-trivial primitive integer solution to the equation $x^{2p}+y^{2q}=z^r$ with $r \nmid a$ and primes $p,q>V_r$, where $V_{r}$ is the same constant as in Theorem~\ref{auxilary result for main result1}.
		By Theorem~\ref{auxilary result for main result1}, there exists an elliptic curve $E^\prime/K$ having full $2$-torsion, good reduction away from $S_{K,2}$ and satisfying certain additional properties. The main idea to prove this theorem is to relate the $j$-invariant $j_{E'}$ of $E'$ in terms of the solutions $(\lambda, \mu)$ of the $S$-unit equation~\eqref{S_K-unit solution} and using the condition~\eqref{assumption for main result1} on $\lambda, \mu$ to get a contradiction.
		
		Since $E^\prime/K$ has full $2$-torsion, it follows that $E^\prime/K$ has a model of the form $$E^\prime: Y^2 = (X-e_1)(X-e_2)(X-e_3),$$
		where $e_1, e_2, e_3 \in K$ are distinct and their cross ratio $\lambda= \frac{e_3-e_1}{e_2-e_1} \in \mathbb{P}^1(K)-\{0,1,\infty\}$. Then  $E^\prime$ is isomorphic (over $\overline{K}$) to an elliptic curve $E_\lambda'$ in the Legendre form:
		\begin{small}		$$E_\lambda' : y^2 = x(x - 1)(x-\lambda) \text{ for } \lambda \in \mathbb{P}^1(K^+)-\{0,1,\infty\} \text{ with}$$ 
			\begin{equation}
				\label{j'-invariant of Legendre form}
				j_{E^\prime} = j(E_\lambda')= 2^8\frac{(\lambda^2-\lambda+1)^3}{\lambda^2(1-\lambda)^2}.
			\end{equation}
		\end{small}
		Since $E^\prime$ has good reduction away from $S_{K,2}$, it follows that $v_\mfq(j_{E^\prime})\geq 0 \text{ for all } \mfq \in P_K \setminus S_{K,2}$ and hence $j_{E^\prime} \in \mcO_{S_{K,2}}$. By \eqref{j'-invariant of Legendre form}, $\lambda$ satisfies a monic polynomial of degree $6$ with coefficients in $\mcO_{S_{K,2}}$, and hence $\lambda \in \mcO_{S_{K,2}}$. But $\frac{1}{\lambda}, \mu:=1-\lambda, \frac{1}{\mu}$ are also solutions of \eqref{j'-invariant of Legendre form}, and hence $\frac{1}{\lambda}, \mu, \frac{1}{\mu} \in \mcO_{S_{K,2}}$. So
		$\lambda, \mu \in \mcO^*_{S_{K,2}}$ and therefore $(\lambda,\ \mu)$ is a solution of the $S_{K,2}$-unit equation~\eqref{S_K-unit solution}.
		Now, rewriting \eqref{j'-invariant of Legendre form} in terms of $\lambda, \ \mu$, we get
		\begin{small}
			\begin{equation}
				\label{j' in terms of lambda and mu}
				j_{E^\prime}= 2^8\frac{(1-\lambda \mu)^3}{(\lambda \mu)^2}.
			\end{equation}
		\end{small}	
		We now use the condition \eqref{assumption for main result1} on  $\lambda, \mu$ to get a contradiction. Let 
		$t:=\max \left\{|v_\mfP(\lambda)|,|v_\mfP(\mu)| \right\} \geq 0$. By \eqref{assumption for main result1}, we have $t \leq 4v_\mfP(2)$. 
		If $t=0$, then $v_\mfP(\lambda)= v_\mfP(\mu)=0$. This gives $v_\mfP(j_{E^\prime})\geq 8v_\mfP(2)>0$, which contradicts Theorem~\ref{auxilary result for main result1}(4). 
		If $t>0$, then using the relation $\lambda + \mu =1$, we have either $v_\mfP(\lambda)=t$ and $ v_\mfP(\mu)=0$, or $v_\mfP(\lambda)=0$ and $v_\mfP(\mu)=t$, or $v_\mfP(\lambda)=v_\mfP(\mu)=-t$. This gives $v_\mfP(\lambda \mu)=t$ or $-2t$. In either way, we get $v_\mfP(j_{E^\prime})\geq 8v_\mfP(2)-2t \geq 0$, which again contradicts Theorem~\ref{auxilary result for main result1}(4). This completes the proof of Theorem~\ref{main result1 for (2p,2q,r)}.
	\end{proof}
	We are now ready to complete the proof of Theorem~\ref{main result2 for (2p,2q,r)} by using Theorems~\ref{auxilary result for main result1} and~\ref{main result1 for (2p,2q,r)}.
	\begin{proof}[Proof of Theorem~\ref{main result2 for (2p,2q,r)}]
		Suppose
		$(a,b,c)$ is a non-trivial primitive integer solution to the equation $x^{2p}+y^{2q}=z^r$ with $r | a$ and primes $p,q>V_r$, where $V_{r}$ is the same constant as in Theorem~\ref{auxilary result for main result1}.
		By Theorem~\ref{auxilary result for main result1}, there exists an elliptic curve $E^\prime/K$ having full $2$-torsion and good reduction away from $S_{K,2r}$.
		Now, using the same arguments as in the proof of Theorem~\ref{main result1 for (2p,2q,r)}, we get 
		$$j_{E^\prime}= 2^8\frac{(1-\lambda \mu)^3}{(\lambda \mu)^2},$$
		where $(\lambda,\ \mu)$ is a solution of the $S_{K,2r}$-unit equation~\eqref{S_{K,2r}-unit solution}. Now, using the condition $\max \left\{|v_\mfP(\lambda)|,|v_\mfP(\mu)| \right\}  \leq 4v_\mfP(2)$ as in \eqref{assumption for main result2}, we get 
		$v_\mfP(j_{E^\prime}) \geq 8 v_\mfP(2)-2v_\mfP(\lambda \mu)  \geq 0$, which contradicts Theorem~\ref{auxilary result for main result1}(4). This completes the proof of Theorem~\ref{main result2 for (2p,2q,r)}.		
	\end{proof}
	%
		%

	\section{Proof of Proposition~\ref{loc crit1 for main result1} and Proposition~\ref{loc crit1 for main result2}}
	\label{section for local criteria}
	\begin{proof}[Proof of Proposition~\ref{loc crit1 for main result1}:]
		Recall that $K= \Q(\zeta_r+ \zeta_r^{-1})$. By hypothesis, $2$ is inert in $K$, hence $\mfP$ is the unique prime ideal of $\mcO_K$ lying above $2$.
		For any solution $(\lambda, \mu) \in (\mcO_{S_{K,2}}^\ast)^2$ to the equation $\lambda+\mu=1$, denote $$m_{\lambda, \mu}:=	\max \left\{ |v_\mfP(\lambda)|,|v_\mfP(\mu)|  \right\}.$$ 
		To prove this proposition, we need to show that every solution $(\lambda, \mu)$ to the $S_{K,2}$-unit equation \eqref{S_K-unit solution} satisfies
		$m_{\lambda, \mu}\leq 4v_\mfP(2)=4.$
		
		We prove this by contradiction. Suppose there exists a solution $(\lambda, \mu)$ to the $S_{K,2}$-unit equation \eqref{S_K-unit solution} satisfies
		$m_{\lambda, \mu} \geq 5.$
		Then there exists a solution $(\lambda', \mu')$ to the $S_{K,2}$-unit equation \eqref{S_K-unit solution} with $\lambda', \mu' \in \mcO_K$ such that 
		$m_{\lambda', \mu'} =m_{\lambda, \mu}$. This is easy because, if $v_\mfP(\lambda) \geq 0$ (respectively, $v_\mfP(\lambda) <0$), then choose $(\lambda', \mu')=(\lambda, \mu)$ (respectively, $(\lambda', \mu')=(\frac{1}{\lambda}, \frac{-\mu}{\lambda})$). Hence $m_{\lambda', \mu'}= \max \left\{ v_\mfP(\lambda'), v_\mfP(\mu')  \right\} \geq 5$.
		Without loss of generality, take $m_{\lambda', \mu'}=v_\mfP(\lambda') \geq 5$. Since the $S_{K,2}$-unit equation \eqref{S_K-unit solution} has only finitely many solutions, we may choose that $(\lambda', \mu')$ is a solution of \eqref{S_K-unit solution} with $v_\mfP(\lambda')$ as large as possible. Since $v_\mfP(\lambda') \geq 5$ and $\lambda'+ \mu'=1$, we get $v_\mfP(\mu')=0$ and hence $\mu' \in  \mcO_K^\ast$. This gives $\mu' \equiv 1 \pmod {\mfP^5}$ and hence $\frac{\mu'-1}{2^5} \in \mcO_K$.
		
		We now show that $\mu'$ is a square in $\mcO_K$ by using the hypothesis $2 \nmid h_K^+$. Suppose $\mu'$ is not a square in $\mcO_K$. Let $\delta:=\frac{1+\sqrt{\mu'}} {2}$ and  $L:=K(\sqrt{\mu'})=K(\delta)$. Then $L$ is a quadratic extension of $K$ and the minimal polynomial of $\delta$ is given by $m_\delta(x)=x^2-x+(\frac{\mu'-1}{4}) \in \mcO_K[x]$ with discriminant $\mu' \in \mcO_K^\ast$. So $L$ is unramified at all finite places of $K$ and $[L:K]=2$, which contradicts $2 \nmid h_K^+$.
		
		Therefore, $\mu'$ is a square in $\mcO_K$. Let $\mu'= \nu^2$, for some $\nu \in \mcO_K^\ast$. Since $\lambda'= 1- \mu'$, we get $\lambda'= (1+ \nu) (1-\nu)$. Let $s_0=v_\mfP(\lambda')$, $s_1=v_\mfP(1 + \nu)$, $s_2=v_\mfP(1-\nu)$. Then $s_0=s_1+s_2 \geq 5$. Since $(1+ \nu)+ (1-\nu)=2$, we get either $s_1=1$ or $s_2=1$. Since $s_0=s_1+s_2 \geq 5$, it follows that exactly one of $s_1,\ s_2$ is equal to $1$.
		If $s_1=1$, then choose $\lambda''=\frac{-(1-\nu)^2}{4\nu}$ and $\mu''= \frac{(1+\nu)^2}{4\nu} $. If $s_2=1$, then choose $\lambda''=\frac{(1+\nu)^2}{4\nu}$ and $\mu''= \frac{-(1-\nu)^2}{4\nu}$. Then $\lambda'', \mu'' \in \mcO_{S_{K,2}}^\ast$ and $\lambda''+ \mu''=1$ with  $v_\mfP(\lambda'')=2(s_0-1)-2=2s_0-4$. Since $s_0 \geq 5$, we get $v_\mfP(\lambda'') > s_0= v_\mfP(\lambda')$, which contradicts the maximality of $ v_\mfP(\lambda')$. This completes the proof of the proposition.
	\end{proof}
	We now prove Proposition~\ref{loc crit1 for main result2} and its proof is similar to \cite[Corollary 7]{M23} and \cite[Corollary 5.7]{JS25}.
	\begin{proof}[Proof of Proposition~\ref{loc crit1 for main result2}:]
		To prove this proposition, it is enough to show that every solution $(\lambda, \mu)$ to the $S_{K,2r}$-unit equation \eqref{S_{K,2r}-unit solution} satisfies
		$	\max \left\{|v_\mfP(\lambda)|,|v_\mfP(\mu)| \right\}\leq 4v_\mfP(2)$, where  $\mfP$ is the unique of $\mcO_K$ lying above $2$. Recall that $r \not\equiv 1 \pmod 8$ and $K= \Q(\zeta_r+ \zeta_r^{-1})$. Then, using the same arguments as in \cite[Theorem 6]{M23} and \cite[Corollary 7]{M23}, we can show that every solution $(\lambda, \mu)$ to the equation \eqref{S_{K,2r}-unit solution} satisfies $\max \left\{|v_\mfP(\lambda)|,|v_\mfP(\mu)| \right\}\leq 4v_\mfP(2)$ (cf. \cite[pp. 13-14]{M23} for more details).
		Finally, by Theorem~\ref{main result2 for (2p,2q,r)}, the proof of the proposition follows.
	\end{proof}
	
	\section*{Acknowledgements} 
	The author thanks Nuno Freitas, Narasimha Kumar and Somnath Jha for various discussions on the solutions of $S$-unit equations over totally real number fields. The author thanks Alejandra Alvarado, Angelos Koutsianas and Christopher Rasmussen for their help with the computations of the solutions of the $S$-unit equation over certain explicit cyclotomic fields. The author thanks Martin Azon and Sunil Pasupulati for their assistance in the use of PARI/GP.


\end{document}